\documentclass[12pt,a4paper,twoside]{article}
\usepackage{amsmath,amsfonts,amsthm,amsopn,color,amssymb,enumitem}
\usepackage{palatino}
\usepackage{graphicx}

\newcommand{\eps}{\varepsilon}

\newcommand{\R}{\mathbb{R}}

\newcommand{\RN}{{\mathbb{R}^N}}

\renewcommand{\le}{\leqslant}
\renewcommand{\ge}{\geqslant}
\renewcommand{\a }{\alpha }

\renewcommand{\d }{\delta }

\newcommand{\g }{\gamma }

\newcommand{\n }{\nabla }
\newcommand{\s }{\sigma }
\renewcommand{\t}{\theta}
\renewcommand{\O}{\Omega}

\newcommand{\G}{\Gamma}

\renewcommand{\L}{\Lambda}

\newcommand{\W}{{\cal W}}

\newcommand{\N}{\mathbb{N}}

\newcommand{\D}{{\mathcal D}}
\newcommand{\Dpn }{{\mathcal D}^{1,p}}
\newcommand{\Dp }{{\mathcal D}^{1,p}(\RN)}
\newcommand{\Dq }{{\mathcal D}^{1,q}(\RN)}
\newcommand{\irn }{\int_{\RN}}

\def\bbm[#1]{\mbox{\boldmath $#1$}}
\newcommand{\beq }{\begin{equation}}
\newcommand{\eeq }{\end{equation}}

\newtheorem{theorem}{Theorem}[section]
\newtheorem{lemma}[theorem]{Lemma}

\newtheorem{definition}[theorem]{Definition}
\newtheorem{proposition}[theorem]{Proposition}
\newtheorem{remark}[theorem]{Remark}

\renewenvironment{proof}{\noindent{\textbf{Proof\quad}}}{$\hfill\square$\vspace{0.2 cm}\\}

\title{{\bf
On a functional satisfying a weak Palais-Smale condition
\footnote{The author is supported by M.I.U.R. - P.R.I.N. ``Metodi
variazionali e topologici nello studio di fenomeni non lineari''}}}

\author{A. Azzollini \thanks{Dipartimento di Matematica, Informatica ed Economia, Universit\`a degli
Studi della Basilicata,  Via dell'Ateneo Lucano 10, I-85100 Potenza,
Italy, e-mail: {\tt antonio.azzollini@unibas.it}}}

\date{}

\begin{document}

\maketitle

\begin{abstract}
In this paper we study a quasilinear elliptic problem whose
functional satisfies a weak version of the well known Palais-Smale
condition. An existence result is proved under general assumptions on the nonlinearities.
\end{abstract}

\section*{Introduction}
 The aim of this paper is to generalize a recent result
obtained in \cite{ADP} concerning the following quasilinear elliptic
problem
\begin{equation}\label{eq}
\left\{
\begin{array}{ll}
-\n \cdot \left[\phi'(|\n u|^2)\n u  \right] +|u|^{\a-2}u =f(u), &
x\in \RN,
\\
u(x) \to 0 , \quad \hbox{as }|x|\to \infty,
\end{array}
\right.
\end{equation}
where $N\ge 2$, $\phi(t)$ behaves like $t^{q/2}$ for small $t$ and
$t^{p/2}$ for large $t$,
 $1< p<q<\min\{p^*,N\}$, $1<\a\le p^* q'/p'$, being $p^*=\frac{pN}{N-p}$
 and $p'$ and $q'$ the conjugate exponents, respectively, of $p$ and
 $q$.\\
 In \cite{ADP} the authors have proved that if $f(t)=|t|^{s-2}t$ grows as
 $t$ goes to $+\infty$ more than $\max\{t^{a-1},t^{q-1}\}$ and less
 than $t^{p^*-1}$ and $\phi \in C^1(\R_+,\R_+)$ satisfies
 \begin{enumerate}[label=(${\rm \Phi}$\arabic*), ref=${\rm \Phi}$\arabic*]
\item \label{phi1} $\phi(0)=0$,
\item \label{phi2} there exists a positive constant $c$ such that
\[
\left\{
\begin{array}{ll}
c t^{\frac p2} \le \phi(t), & \hbox{if }t\ge 1,
\\
c t^{\frac q2} \le \phi(t), & \hbox{if }0\le t\le 1,
\end{array}
\right.
\]
\item \label{phi3} there exists a positive constant $C$ such that
\[
\left\{
\begin{array}{ll}
\phi(t) \le C t^{\frac p2}, & \hbox{if }t\ge 1,
\\
\phi(t) \le C t^{\frac q2}, & \hbox{if }0\le t\le 1,
\end{array}
\right.
\]
\item \label{phi4} the map $t \mapsto \phi(t^2)$ is strictly convex,
\item \label{phi5} there exists $0<\mu<1$ such that
\[
\phi'(t)t\le \frac{s\mu}{2}\phi(t), \quad \hbox{for all }t\ge 0,
\]
\end{enumerate}
 then the problem possesses infinitely many solutions. As remarked in that paper, the same result can be obtained
 if we assume more general hypotheses on the
 nonlinearity $f$. In particular, apart from the local assumptions related with the
 behaviour at 0 and at infinity, it is required the following global condition
\[
0<\t F(t)\le f(t)t, \quad \hbox{for all }t> 0,
\]
where $\t>\a$ and $F(t)=\int_0^tf(z)\,dz$. This assumption, known as the Ambrosetti-Rabinowitz
condition, (AR) in short, is quite classical in the field of
critical points theory and typically occurs when we try to prove the
boundedness of the
Palais-Smale sequences related with the functional of the action.\\
However, some papers have shown that there are many situations in
which (AR) can be successfully bypassed. In \cite{BL1}, for
instance, the equation
\begin{equation}
\left\{
\begin{array}{ll}    \label{eq:BL}
        -\Delta u = g(u)\\
u(x) \to 0 , \quad \hbox{as }|x|\to \infty,
\end{array}
\right.
\end{equation}
is solved without (AR) in two steps. First the authors reduce the
problem to that of minimizing a constrained (bounded below)
functional, obtaining a solution to the equation, up to a Lagrange
multiplier. Then they exploit the behaviour of the equation with
respect to the rescaling to make the Lagrange multiplier disappear.

More recently, in \cite{S} and \cite{J} it has been shown a method,
named {\it monotonicity trick}, which exploits the differentiability
a.e. of the monotone functions to get bounded Palais Smale sequences
for functionals related with approximating equations.\\ If there is
no problem of compactness, this method allows us to get a Palais
Smale sequence constituted by solutions of approximating equations.
Afterwards, getting some more information on the elements of the
sequence using the fact that they solve an equation, we could prove
the boundedness of the Palais-Smale sequence. Usually, the
additional information we look for is the well known Pohozaev
identity, an equality satisfied by sufficiently regular solutions of
elliptic equations in the divergence form.

In our situation, a different approach is required. Consider the
problem
\begin{equation}\label{eq2}
\left\{
\begin{array}{ll}
-\n \cdot \left[\phi'(|\n u|^2)\n u  \right]  =g(u), & x\in \RN,
\\
u(x) \to 0 , \quad \hbox{as }|x|\to \infty,
\end{array}
\right.
\end{equation}
where we will assume on $g$ hypotheses similar to those in
\cite{BL1}.

Observe that, since $\phi'$ is not homogeneous, we can not proceed
as in \cite{BL1}. On the other hand, also the monotonicity trick
does not seem to be of use. Indeed, since no regularity result on
the solutions of \eqref{eq2} is available, we can not obtain a Pohozaev identity in
a standard way. To overcome these difficulties, we use a result
contained in \cite{HIT}, where an alternative way of approaching
\eqref{eq:BL} is showed. The method
presented consists in adding a dimension to the space where the
problem is set, and constructing a Palais-Smale sequence for a
suitable auxiliary functional defined in this new space. Such a
technique, which we call {\it the adding dimension technique},
permits to get additional information on a Palais-Smale sequence of
the original functional and, possibly, to prove it is bounded. We remark the fact that this
method does not require any regularity assumption on the solutions of the equations.

It is worthy of note that, differently from the functional related
with \eqref{eq:BL}, the functional of the action associated with
\eqref{eq2} will be defined on a particular Orlicz-Sobolev space.
Treating with this space carries some more complications when we try
to solve \eqref{eq2} with a nonlinearity modeled on that in
\cite{BL1}. To explain better what we mean, we list our assumptions
on $g$ and state the main result.

Suppose that $g$ is a continuous function satisfying the following hypotheses
            \begin{enumerate}[label=(g\arabic*), ref=(g\arabic*)]
                \item\label{g1} $-\infty \le\limsup_{s\to +\infty} g(s)/s^{p^*-1}\le 0$, with $p^*=pN/(N-p)$;
                \item \label{g2} $-\infty <\liminf_{s\to 0^+} g(s)/{s^{\a-1}}\le \limsup_{s\to 0^+} g(s)/{s^{\a-1}}=-m<0,$
                 for $1<\a<p^*$;
            \end{enumerate}
and the following Berestycki-Lions condition
            \begin{itemize}
              \item [(BL)]\label{BL} there exists $\zeta>0$ such that $G(\zeta):=\int_0^\zeta g(s)\,d
              s>0$.
            \end{itemize}
We remark that, reading (AR) as a growth condition on $f,$ it is not
difficult to see that, if we set $g(u)=-|u|^{\a-2}u +f(u)$,
condition (AR) implies (BL). We will prove the following result
    \begin{theorem}\label{main}
        If $N\in\N$ with $N\ge 2,$ $1< p<q<\min\{N,p^*\}$,  $p^*N'/p'\le\a\le p^*q'/p'$,
        and (\ref{phi1}-\ref{phi4}), \ref{g1},\ref{g2}, (BL) hold, then problem \eqref{eq2}
        possesses at least
        a nonnegative radially symmetric solution.
    \end{theorem}

Comparing the main result in \cite{ADP} with ours, we note that the
prize we have to pay to generalize the nonlinearity is a more
restrictive assumption on $\a$, which we require is not too close to
1. This fact arises from a significant difference between the
classical embedding results known for Sobolev spaces, and the
embedding results
available for the Orlicz-Sobolev space where we set our problem.\\
To clarify this point, we recall a well known fact. Consider
$\mathcal D(\RN),$ the set of all $C^{\infty}$ function in $\RN$
with a compact support and set $1<p<N$. Define the norm
$\|\cdot\|_{\Dpn}$ such that for all $u\in\mathcal D,$
    \begin{equation*}
        \|u\|_{\Dpn}=\left(\irn |\n u|^p\,dx\right)^{\frac1p}
    \end{equation*}
and set
    \begin{equation}\label{eq:no}
        \Dp=\overline{\mathcal D(\RN)}^{\|\cdot\|_{\Dpn}}.
    \end{equation}
It is well known that $\Dp\hookrightarrow L^{p^*}(\RN),$ so that, for any $u\in \Dp$,
    \begin{equation}\label{eq:emb}
        \left(\irn|u|^{p^*}\, dx\right)^{\frac 1{p^*}}\le C\left(\irn|\n u|^p\,dx\right)^{\frac 1p}.
    \end{equation}
If $1<\a$ and $\|\cdot\|_\a$ is the usual Lebesgue norm, we define
    \begin{equation}\label{eq:V}
        {\mathcal V}=\overline{\mathcal
        D(\RN)}^{\|\cdot\|_{\Dpn}+\|\cdot\|_\a}.
    \end{equation}
Of course, since ${\mathcal V}$ is continuously embedded in $\Dp$,
inequality \eqref{eq:emb} holds true for any $u\in {\mathcal V}.$
Observe that, if $\phi(t)=t^{\frac p2},$
the space ${\mathcal V}$ would be a nice set to study problem \eqref{eq}.\\
If we want to proceed analogously in our situation, we have to
construct a space ${\mathcal W}$ taking into account \eqref{phi2} and
\eqref{phi3}. We have to substitute the norm $\|\cdot\|_{\Dpn}$, with
a Luxemburg norm to be computed on $\n u$. Assumptions \eqref{phi2}
and \eqref{phi3} suggest to use the norm of the space $L^p(\RN)+L^q(\RN)$, which
we call $\|\cdot\|_{p,q}$ and to replace Sobolev space $\Dp$ with
the Orlicz-Sobolev
    \begin{equation*}
    {\mathcal D}^{1,p,q}(\RN)=\overline{\mathcal
    D(\RN)}^{\|\n \cdot\|_{p,q}}.
    \end{equation*}
Unfortunately, the analogy with the Sobolev spaces stops here.
Indeed, since $\Dp$ and $\Dq$ are continuously embedded in $
{\mathcal D}^{1,p,q}(\RN)$, there is no continuous embedding of
${\mathcal D}^{1,p,q}(\RN)$ in any Lebesgue space (see Remark \ref{rem:emb}).\\
However, in \cite{ADP} it has been proved that, if we define the
analogous of ${\mathcal V}$ in the following way
$${\mathcal W}=\overline{\mathcal D(\RN)}^{\|\n\cdot\|_{p,q}+\|\cdot\|_\a},
\quad \hbox{ for }1<\a<p^*q'/p' $$ then the following inequality
holds true
    \begin{equation}\label{eq:ineq}
        \|u\|_{p^*}\le C (\|u\|_\a+\|\n u\|_{p,q}),\quad\hbox{ for all } u\in {\mathcal W}.
    \end{equation}
From this, we deduce that ${\mathcal W}\hookrightarrow L^r(\RN)$ for
any $\a\le r\le p^*$, even if, differently from ${\mathcal V}$, it
is not possible to control the $L^{p^*}(\RN)$ norm just with the
$L^p(\RN)+L^q(\RN)$ norm of the gradient. This fact translates to a
technical difficulty in proving the Palais Smale condition.
Precisely, we will show that the functional of the action satisfies
a compactness condition weaker than the Palais-Smale if $\a$ is not
too close to $1$.

Finally, we point out that, since we do not require assumption
\eqref{phi5}, our existence result holds for function $\phi$ more
general than those treated in \cite{ADP}.

The paper is so organized: in section \ref{sec1} we will introduce
the functional setting, and the related properties we will use in
our variational approach to the problem. For the most part, the
results contained in this section are proved in \cite{ADP} and \cite{BPR} so we
only recall them. In section \ref{sec2} we will define a new
weakened version of the Palais-Smale condition, and we will
introduce the definition of a particular type of Palais-Smale
sequences. Finally, in section \ref{sec3}, we will prove our main
result by means of the adding dimension technique introduced in
\cite{HIT}.

\subsection*{Notation}
\begin{itemize}

\item $\mathbb{K}=\R$ or $\mathbb{K}=\RN$ according to the case.
\item If $r>0$, we denote by $B_r$ the ball of center $0$ and radius $r$.
\item If $\O\subset \RN$, then $\O^c=\RN \setminus \O$.
\item Everytime we consider a subset of $\RN,$ we assume it is measurable and we denote by $|\cdot|$ its measure.
\item We denote by $\D$ the space of all functions in $C^{\infty}(\RN)$ with compact support.
\item If $\O\subset\RN$, $\tau \ge 1$ and $m\in\N^*$,  we denote by $L^\tau (\O)$ the Lebesgue space $L^\tau (\O,\mathbb{K})$, by $\|\cdot\|_{L^\tau(\O)}$ its norm ($\|\cdot\|_\tau$ if $\O=\RN$) and by $W^{m,\tau}(\O)$ the usual Sobolev spaces.
\item $C$ and $c$ will denote generic constants which would change
from line to line.

\end{itemize}

\section{The functional setting}\label{Antoniofigo}\label{sec1}
This section is devoted to the construction of the functional
setting.
\\
As a first step, we have to recall some known facts on the sum of
Lebesgue spaces.
\begin{definition}
Let $1<p<q$ and $\O\subset \RN$. We denote with $L^p(\O)+L^q(\O)$
the completion of $C_c^\infty (\O,\mathbb{K})$ in the norm
\begin{equation}
\label{eq:pq} \| u\|_{L^p(\O)+L^q(\O)}=\inf\left\{\| v\|_p+\|
w\|_q\;\vline\; v \in L^p(\O), w\in L^q(\O), u= v+w\right\}.
\end{equation}
We set $\| u\|_{p,q}=\|u\|_{L^p(\RN)+L^q(\RN)}$.
\end{definition}

In the next proposition we give a list of properties that will be
useful in the rest of the paper. The reader can find the proofs in
\cite{ADP} and \cite{BPR}

\begin{proposition}\label{pr:lplq}
Let $\O\subset\RN$, $u\in L^p(\O)+L^q(\O)$ and $\Lambda_{u}=\left\{
x\in\O \; \vline \; | u(x)| > 1 \right\}$. We have:
\begin{enumerate}[label=({\rm \roman*}), ref=({\rm \roman*})]
\item \label{en:op} if $\O'\subset\O$ is such that $|\O'|<+\infty$, then $u\in L^p(\O')$;
\item \label{en:oinf} if $\O'\subset\O$ is such that $u\in L^{\infty}(\O')$, then $u\in L^q(\O')$;
\item \label{en:lu} $|\Lambda_{u}|<+\infty$;
\item \label{en:ulu} ${u} \in L^p(\Lambda_{u}) \cap L^q(\Lambda^c_{u})$;
\item \label{en:dual} $L^p(\O)+L^q(\O)$ is reflexive and $(L^p(\O)+L^q(\O))'=L^{p'}(\O)\cap L^{q'}(\O)$;
\item \label{en:lupq} $\| u \|_{L^p(\O)+L^q(\O)} \le \max \left\{\| u\|_{L^p(\Lambda_{\bf u})}, \| u\|_{L^q(\Lambda^c_{\bf u})} \right\}$;
\item \label{en:ooc} if $B\subset\O$, then  $\| u\|_{L^p(\O)+L^q(\O)}\le \| u\|_{L^p(B)+L^q(B)}+\| u\|_{L^p(\O\setminus B)+L^q(\O\setminus B)}$.
\end{enumerate}
\end{proposition}

We can now define the Orlicz-Sobolev space where we will study our
problem.

\begin{definition}
We assume the following definition
$${\mathcal D}^{1,p,q}(\RN)=\overline{\mathcal
    D(\RN)}^{\|\n \cdot\|_{p,q}}.$$
Moreover, if $\a >1$, we denote with $\mathcal{W}$ the following
space
\[
\mathcal{W}=\overline{\mathcal{D}(\RN)}^{\|\cdot\|}.
\]
where $\|\cdot\|=
\|\cdot\|_\a+\|\n\cdot
\|_{p,q}.$
\end{definition}

We again refer to \cite{ADP} for the proofs of the following
propositions and theorems on the space $\mathcal{W}$

\begin{proposition}\label{pr:rifle}
$(\mathcal{W},\|\cdot\|)$ is a reflexive Banach space.
\end{proposition}

\begin{proposition}
    If $u\in \mathcal{W},$ then it verifies the following inequality
        \begin{equation}\label{eq:brez2}
            \|u\|_{p^*}^t \le C \left(
            \|u\|_{p^*}^{t-1} + \|u\|_\a^{t-1} \right)
            \|\n u\|_{p,q}
        \end{equation}
    where $C$ is a positive constant which does not depend on $u$ and
    $t> 1$ satisfies the equality $\frac t{t-1}=\frac{p(N-1)}{N(p-1)}.$
\end{proposition}

\begin{theorem}\label{th:immersione}
If $1<p<\min\{q,N\}$ and $1< p^*\frac{q'}{p'}$ then, for every
$\a\in\left(1,p^*\frac{q'}{p'}\right]$, the space
$(\mathcal{W},\|\cdot\|)$ is continuously embedded into
$L^{p^*}(\RN)$.
\end{theorem}

\begin{remark}
By interpolation we have that $(\mathcal{W},\|\cdot\|)$ is
continuously embedded into $L^\tau(\RN)$ for any $\tau\in [\a,p^*]$.
\end{remark}

\begin{remark}\label{rem:emb}
    The normed space $({\mathcal W},\|\n\cdot\|_{p,q})$
    is not continuously embedded in any Lebesgue
    space. Indeed, consider $\varphi\in\mathcal D(\RN)$ and for any
    $t>0$ set $\varphi_t=t^{\frac{p-N}p}\varphi(\frac{\cdot}{t})$.
    Of course for any $t>0$ the function
    $\varphi_t\in\W$ and we have that
        \begin{align*}
            \irn |\n\varphi_t|^p\,dx&=\irn|\n\varphi|^p\,dx\\
            \irn |\varphi_t|^{r}\,dx&=t^{\frac{(p-N)r}{p}+N}\irn|\varphi|^{r}\,dx
        \end{align*}
    Since $\|\n\varphi_t\|_{p,q}\le\|\n\varphi_t\|_p,$ we deduce
    that the family $(\varphi_t)_{t>0}$ is bounded in $({\mathcal
    W},\|\n\cdot\|_{p,q})$.
    On the other hand, if $r\neq p^*,$ we can make the Lebesgue norm as large as we
    want just taking large $t,$ if $1<r<p^*$ and small $t,$  if
    $p^*<r.$ So $(\mathcal W,\|\n\cdot\|_{p,q})$ does not embed in any $L^r(\RN),$ for $1<r\neq p^*.$
    We see that $({\mathcal W},\|\n\cdot\|_{p,q})$ does not
    embed even in $L^{p^*}(\RN)$ just observing that, if for any
    $s>0$ we set
    $\varphi_s=s^{\frac{q-N}q}\varphi(\frac{\cdot}{s}),$ then
        $$
            \sup_{s>0}\|\n\varphi_s\|_{p,q}\le\sup_{s>0}\|\n\varphi_s\|_{q}<+\infty,\qquad
            \sup_{s>0}\|\varphi_s\|_{p^*}=+\infty.
        $$
\end{remark}

Now we define the functional of the action related with our problem.
For any $u\in\mathcal{W}$ we set (from now on, we omit the symbol {\it dx} in the integration)
    \begin{equation*}
      J(u)=\frac 12 \irn \phi (|\n u|^2) -\irn G(u),
    \end{equation*}
where $G:\R\to\R$ is defined as in assumption (BL). To make the
functional well defined and $C^1,$ we modify $g$ according to the
following two possibilities:
    \begin{itemize}
        \item[{\it 1st case:}] $\liminf_{s\to+\infty}\frac{g(s)}{s^{p^*-1}}=0.$

        Then we define $\tilde g=g_1-g_2$ where
           \begin{equation*}
      g_1(s)=\left\{
        \begin{array}{ll}
                (g(s)+ms^{\a-1})^+, & \hbox{if }s\ge0,
                \\
                0, & \hbox{if }s<0,
      \end{array}
      \right.
    \end{equation*}
    and
     \begin{equation*}
     g_2(s)=\left\{
        \begin{array}{ll}
                g_1(s)-g(s), & \hbox{if }s\ge0,
                \\
                -g_2(-s), & \hbox{if }s<0,
      \end{array}
      \right.
    \end{equation*}
    \item[{\it 2nd case:}]$\liminf_{s\to+\infty}\frac{g(s)}{s^{p^*-1}}<0.$

    Then there exist $\eps>0$ and an increasing sequence of positive numbers
    $(s_n)_n$ such that $g(s_n)\le -\eps s_n^{p^*-1}.$ By continuity,
    certainly there exists $s_0>0$ such that $g(s_0)+ms_0^{\a-1}=0.$
    We set $\tilde g=g_1-g_2$ where
    \begin{equation*}
      g_1(s)=\left\{
        \begin{array}{ll}
                (g(s)+ms^{\a-1})^+, & \hbox{if }s\in [0,s_0],
                \\
                0, & \hbox{if }s\in [0,s_0]^{c},
      \end{array}
      \right.
    \end{equation*}
    and
     \begin{equation*}
     g_2(s)=\left\{
        \begin{array}{ll}
                g_1(s)-g(s), & \hbox{if }s\in [0,s_0],
                \\
                ms^{\a-1}, & \hbox{if }s_0<s
                \\
                -g_2(-s), & \hbox{if }s<0.
      \end{array}
      \right.
    \end{equation*}
    \end{itemize}
Since $\tilde g$ satisfies
    \begin{itemize}
        \item[(g3)]
        $\lim_{s\to\infty} \frac{|\tilde g(s)|}{|s|^{p^*-1}}=0$,
    \end{itemize}
by \cite{BPR} and Theorem \ref{th:immersione} we can prove $J$ is
well defined and $C^1$ in $\mathcal W$ if we replace $g$ with
$\tilde g.$ On the other hand, we point out that, if  $u\in\mathcal
W$ solved equation \eqref{eq2} with $\tilde g$ in the place of $g,$
then $0\le u$ and, if the second case occurred, then we also would
have $u\le s_0.$ As a consequence, we deduce that no loss of
generality would arise supposing that $g$ is defined as $\tilde g$
and (g3) holds.

Some simple computations show that for functions $g_1$ and $g_2$ the following
properties hold
    \begin{itemize}
        \item[(i)] $g_1$ and $g_2$ are nonnegative in $\R_+,$
        \item[(ii)] $g=g_1-g_2,$
        \item[(iii)] $\lim_{t\to\infty}\frac{g_1(t)}{|t|^{p^*-1}}=0,$ $\lim_{t\to 0^+}\frac{g_1(t)}{t^{\a-1}}=0,$
        \item[(iv)] there exists a positive constant $a$ such that $a t^{\a-1}\le g_2(t)$, for any $t\in\R_+,$
        \item[(v)] for any $\eps>0$ there exists $C_\eps>0$ such that $g_1(t)\le \eps g_2(t)+C_\eps t^{p^*-1}$,
        for any $t\in\R_+.$
    \end{itemize}
Once we have set $G_i(z):=\int_0^z g_i(s)\,ds>0$ for $i=1,2,$ we
have that the functional can be written
    \begin{equation*}
        J(u)=\frac 12 \irn \phi (|\n
        u|^2) +\irn G_2(u)-\irn G_1(u),
    \end{equation*}
In order to have compactness, we introduce a symmetry requirement on
our space.
\begin{definition}
Let us denote with
\[
\D(\RN)_{{\rm rad}}=\{u\in  \D(\RN)\mid u \hbox{ is radially
symmetric}\},
\]
and let $\mathcal{W}_r$ be the completion of $\D (\RN)_{{\rm rad}}$
in the norm $\|\cdot\|$, namely
\[
 \W_r=\overline{\D(\RN)_{{\rm rad}}}^{\|\cdot\|}.
\]
\end{definition}

\begin{remark}
        Observe that if $1\le \a\le p^*$, the set $\mathcal{W}_r$
        is included in $L^s(\RN)$ for any $s\in[\a,p^*].$\\
        Indeed, take $u\in\mathcal{W}_r,$ and
        consider the set $\Lambda_{\n u}.$
        Since $\|\n u\|_{p,q}<+\infty$, certainly $\n u\in L^p(\Lambda_{\n u}).$
        On the other hand, since $u\in L^\a(\RN),$ we have that $u\in L^\a(\Lambda_{\n
        u}).$ So, if we define $E(\Omega)=\{v\in L^\a(\Omega)\mid\n u\in
        L^{p}(\Omega)\}$, then $u\in E(\Lambda_{\n u}).$\\
        By symmetry of $u$, the set $\Lambda_{\n u}$ has a smooth
        boundary so, by standard arguments (see for example \cite{A}), there exists a
        continuous extension operator $T: E(\Omega)\to E(\R^N).$
        Then embedding theorems hold in the domain $\Lambda_{\n u}$
        so we deduce that $u\in L^s(\Lambda_{\n u})$ for any $s\in
        [\a,p^*]$. Analogously $u\in L^s(\Lambda_{\n u}^c),$ for any
        $s\in[\a,q^*].$ Since $p^*< q^*,$ we conclude.\\
        At the
        present stage of our knowledge, we do not know if, for $p^*\frac{q'}{p'}<\a< p^*$,
        these
        embeddings are also continuous.

    \end{remark}

The following compactness result holds.

\begin{theorem}
\label{th:compact} If $1<p<q<N$ and $1< \a\le p^*\frac{q'}{p'},$
then the functionals
    \begin{align*}
        u\in\mathcal W_r&\mapsto \irn G_1(u)\\
        u\in\mathcal W_r&\mapsto \irn g_1(u)u
    \end{align*}
are weakly continuous.
\end{theorem}
\begin{proof}
    We prove the weak continuity of the first functional. By Lemma
    2.13 in \cite{ADP}, for any $u\in\mathcal W_r,$
        \begin{equation}\label{eq:unifdec}
            | u(x) | \le \frac{C}{|x|^\frac{N-q}{q}}\| \n u
            \|_{p,q}, \quad\hbox{for }|x|\ge 1.
        \end{equation}
    Now, consider a sequence $(u_n)_n$ in $\mathcal W_r$ weakly
    convergent to $u_0.$ By Theorem 2.11 in \cite{ADP}, $(u_n)_n$
    possesses a subsequence strongly convergent to $u_0$ in $L^{\tau}(\RN)$ for any
    $\tau\in ]\a,p^*[$. So, up to subsequences, we can assume that
    $(u_n)_n$ converges almost everywhere to $u_0.$\\
    Set $P(t)=G_1(t),$ $Q(t)=|t|^\a+|t|^{p^*},$ $v=G_1(u_0).$ By
    property
    (iii) of the function $g_1$, remark 1.7 and \eqref{eq:unifdec},
    we can apply the
    Strauss compactness Lemma in the version as it appears in
    \cite{BL1} and conclude. In a similar way we prove the rest of
    the statement.

\end{proof}

\section{A weak Palais-Smale condition}\label{sec2}

As it is well known, the Palais-Smale condition is a compactness
property related to a functional defined on a Banach space.\\
It states as follows: let $I:E\to\R$ be a $C^1$ functional on the Banach space $E$ and $c\in\R.$
If for any given $(x_n)_n$ in $E$ such that $I(x_n)\to c$ and $I'(x_n)\to 0$ there exists a converging subsequence of $(x_n)_n,$ we say that $I$ satisfies the Palais Smale condition at the level $c.$\\
 Usually,
in the calculus of variation, testing the Palais Smale condition
consists in two steps: first we check if every Palais Smale sequence (namely a sequence verifying the previous assumptions)
is bounded, second we handle with compact embedding theorems to
prove strong convergence (up to a subsequence) in the Banach space. Many times it happens
that the main problem in verifying Palais-Smale condition is related
with the first step. In such cases, one tries to prove that the functional satisfies at least a weakened version of
the Palais-Smale condition, and look for the existence of at least one sequence to which that condition can be applied.\\
In this direction, we recall, for example, the well-known Cerami version of the Palais-Smale condition (see \cite{Cer}),
and the problem in \cite{BBF} where this condition is applied.\\
Here we introduce a new version of a weakened Palais-Smale
condition.
    \begin{definition}\label{def}
        Suppose that $(E,\|\cdot\|_E)$ and
        $(F,\|\cdot\|_F)$ are two Banach spaces such that $(E,\|\cdot\|_E)\hookrightarrow (F,\|\cdot\|_F)$.
        A functional $I\in C^1(E,\R)$ satisfies a weak Palais-Smale condition with respect to $E$ and $F$
        if for any sequence $(x_n)_n$ in $E$  such that
            \begin{enumerate}
                \item $I(x_n)$ is bounded,
                \item $I'(x_n)\to 0$ in $E',$
                \item $(\|x_n\|_F)_n$ is bounded,
            \end{enumerate}
        there exists a converging subsequence (in the topology of
        $E$).
    \end{definition}

    \begin{remark}\label{rem:bl}
        Consider the functional of the action related with the equation
$$-\Delta u=g(u)$$
where $g$ is as in \cite{BL1}. After having produced a suitable
modification of the function $g$, we can see that finite energy
solutions of the equation are critical points of
$$I(u)=\frac 1 2\irn |\n u|^2-\irn G(u),\quad u\in H^1_r(\RN),$$
being $G$ a primitive of $g.$ The properties on $g_1$ and $g_2$
listed in $(i)\ldots (v)$ hold, except that we
have to replace $\alpha$ with $2$ and $p^*$ with $2^*.$\\
We show that $I$ verifies a weak Palais-Smale condition  with
respect to $H^1_r(\RN)$ and $\D^{1,2}_r(\RN)$, each one provided
with its natural norm. Indeed suppose $(u_n)_n$ is a Palais Smale
sequence for which $(\|\n u\|_2)_n$ is bounded. Since $I(u_n)$ is
bounded, for a certain $M>0$ and any $\eps \in ]0,1[$ there exists a
suitable $C_\eps>0$ for which we have
\begin{align*}
    \frac 12 \irn |\n u_n|^2+\irn G_2(u_n)&\le\irn G_1(u_n) +M\\
    &\le \eps \irn G_2(u_n) +C_\eps\irn |u_n|^{2^*}+M\\
    &\le \eps \irn G_2(u_n) +C_SC_\eps\left(\irn |\n u_n|^2\right)^{\frac{2^*}{2}}+M
\end{align*}
where $C_S^{\frac 1 {2^*}}$ is the Sobolev constant for the embedding $\D^{1,2}(\RN)\hookrightarrow L^{2^*}(\RN).$ \\
Then
    \begin{equation*}
        \frac 12 \irn |\n u_n|^2+(1-\eps)\irn G_2(u_n)\le C_SC_\eps\left(\irn |\n u_n|^2\right)^{\frac{2^*}{2}}+M
    \end{equation*}
and, since $(\|\n u_n\|_2)_n$ is bounded, we conclude that $(\|u_n\|_2)_n$ is also bounded since we have
    \begin{align*}
        a(1-\eps)\irn |u_n|^2&\le \frac 12 \irn |\n u_n|^2+(1-\eps)\irn G_2(u_n)\\
        &\le C_SC_\eps\left(\irn |\n u_n|^2\right)^{\frac{2^*}{2}}+M.
    \end{align*}
At this point the arguments are quite standard: we extract a weakly
convergent (in $H^1$-norm) subsequence and we use radial symmetry of
functions in our space and a Strauss compactness lemma to find a
strong convergent sequence.
    \end{remark}

Set
    \begin{equation}\label{eq:D1pq}
        {\mathcal Y}_r=\overline{\D(\RN)_{{\rm rad}}}^{\|\n
        \cdot\|_{p,q}}.
    \end{equation}
We have the following result
    \begin{theorem}\label{le:ps}
        Under the assumptions of Theorem \ref{main}, the functional
        $J$ satisfies a weak Palais-Smale condition, with respect
        to ${\mathcal W}_r$ and ${\mathcal Y}_r$.
    \end{theorem}
    \begin{proof}
        Suppose $(u_n)_n$ is a sequence of functions in ${\mathcal
        W}_r$ such that $1, 2$ and $3$ of definition \ref{def}
        hold.\\
        We first prove that the sequence is bounded.
        By computations analogous to those in remark \ref{rem:bl}, there exist $M>0$ and $C_1> 0,$ such that
        \begin{align}\label{eq:bound}
    \frac 12 \irn \phi(|\n u_n|^2)+(1-\eps)\irn G_2(u_n)&\le C_1\irn
    |u_n|^{p^*}+M.
        \end{align}
    Now, if $(u_n)_n$ is bounded in the $L^{p^*}-$norm, we have
    concluded.\\
    By \eqref{eq:brez2} and 3 of definition
        \ref{def},
            \begin{equation}\label{eq:estim}
                \|u_n\|_{p^*}^t \le C_2 \left(
            \|u_n\|_{p^*}^{t-1} + \|u_n\|_\a^{t-1} \right)
            \end{equation}
        for some $C_2 >0.$
    Suppose that $(\|u_n\|_{p^*})_n$ diverges (up to a subsequence). Then, by \eqref{eq:estim},
     certainly there exists a constant $C$ such that, definitely,
        \begin{equation}\label{eq:estim2}
            \|u_n\|_{p^*}\le C\|u_n\|_{\a}^{\frac{t-1}{t}}.
        \end{equation}
    Comparing \eqref{eq:bound} and \eqref{eq:estim2},
    taking into account that $\frac t{t-1}=\frac{p'}{N'}$ and $a\|u_n\|_{\a}^{\a}\le\irn G_2(u_n),$
    we have, for some positive constant $C$,
        \begin{equation*}
            \|u_n\|_{\a}^{\a}\le C\|u_n\|_{\a}^{\frac{N'p^*}{p'}}
        \end{equation*}
    and then, since $\a>\frac{N'p^*}{p'},$ the sequence $(\|u_n\|_\a)_n$ is bounded.\\
    Therefore, by Proposition \ref{pr:rifle} and Theorem \ref{th:compact}, there exists $u_0\in \W_r$ such that, up to subsequences,
\begin{align}
&u_n\rightharpoonup u_0, \quad \hbox{weakly in
}\W_r,\label{eq:debole}
\\
&\irn G_1(u_n) \to \irn G_1(u_0),  \label{eq:forte}
\\
&\irn g_1(u_n)u_n\to\irn g_1(u_0)u_0,
\end{align}
and, by \cite[Theorem 2.11]{ADP},
    $$u_n \to u_0, \quad \hbox{a.e. in }\RN.$$
    From this point till the end, the proof follows the scheme of Proposition 3.3
    in \cite{ADP} and of Proposition 2 in \cite{BL2}, step $9.1 c$ (see also \cite[Lemma 3.5]{APindiana}).
    We point out only the key passages.
By (\ref{eq:debole}) and arguing as in \cite[page 208]{LL}, we have
\begin{align}
&\n u_n\rightharpoonup \n u_0, \quad \hbox{weakly in
}L^p(\RN)+L^q(\RN), \label{eq:debolelpq}
\\
&u_n\rightharpoonup u_0, \quad \hbox{weakly in }L^\a(\RN).
\label{eq:debolel2}
\end{align}
As in \cite{APindiana} we prove that for any $z\in C^\infty_0(\RN),$
we have
    \begin{equation}
        \irn g_i(u_n)z\to\irn g_i(u_0)z\quad
        i=1,2.\label{eq:compdual}
    \end{equation}
    Set
        \[
            A_1(u)=\frac 12 \irn \phi(|\n u|^2), \quad A_2(u)= \irn
            G_2(u),\quad B(u)=\irn G_1(u).
        \]
    By \eqref{eq:compdual}, and since $(u_n)_n$ is bounded in $\mathcal W_r$,
    $J'(u_n)\to 0$ in $\mathcal W_r'$ implies that $A_1'(u_n)\to B'(u_0)-A_2'(u_0)$ in $\mathcal
    W_r'.$
    By convexity, we have
        \begin{align*}
            A_1(u_n) &\le A_1(u_0)+A_1'(u_n)[u_n-u_0]
        \end{align*}
    and then, passing to the limit,
        \begin{equation*}
            \limsup_n A_1(u_n) \le A_1(u_0).
        \end{equation*}
    Since, by weak lower semicontinuity of $A_1$ we also have
    \[
        A_1(u_0)\le \liminf_n A_1(u_n),
    \]
    we conclude that
    \begin{equation}\label{eq:convpq}
        \lim_n A_1(u_n)=A_1(u_0).
    \end{equation}
    By (\ref{eq:debolelpq}) and (\ref{eq:convpq}), we can deduce (see \cite{DS})
        \begin{equation*}
            \n u_n \to \n u_0,   \quad \hbox{in }L^p(\RN)+L^q(\RN).
        \end{equation*}
    Moreover, since
        \begin{align*}
            \lim_n\irn g_2(u_n)u_n&=\lim_n \left(\irn g_1(u_n)u_n-\irn \phi'(|\n u_n|^2)|\n
            u_n|^2\right)\\
            &= \irn g_1(u_0)u_0 - \irn \phi'(|\n u_0|^2)|\n
            u_0|^2\\
            &= \irn g_2(u_0)u_0,
        \end{align*}
    we are able also to prove that $u_n\to u_0$ in $L^\a(\RN)$ and
    we conclude.
    \end{proof}

\section{Proof of the main Theorem}\label{sec3}

    In view of Lemma \ref{le:ps}, we have just to find a level for which we can
    find a Palais-Smale sequence satisfying the boundedness assumption 3 in the Definition
    \ref{def}.\\
    \begin{lemma}\label{le:gamma}
        The set
        \begin{equation*}
            \Gamma:=\{\gamma\in C([0,1],\mathcal W_r)\mid \gamma(0)=0,
            I(\gamma(1))< 0\}
        \end{equation*}
    is nonempty.
    \end{lemma}
    \begin{proof}
        Starting from the function $z\in\mathcal D(\RN)$ for which $\irn G(z)>0$ (the existence of such a
        function is proved in \cite{BL1}), the proof is
        standard. Indeed consider $z_l(\cdot)=z(\cdot/l)$ for a value of
        $l>0$ to be established and compute
                \begin{align*}
                    J(z_l)&\le C_1 \int_{\L_{\n z_l}^c} |\n z_l|^q
                    +  C_2 \int_{\L_{\n z_l}} |\n z_l|^p -\irn G(z_l)\\
                        &\le C \irn |\n z_l|^p -\irn G(z_l)=
                        Cl^{N-p}\irn |\n z|^p-l^N\irn G(z).
                \end{align*}
            We deduce that $J(z_l)<0$ if $l$ is sufficiently large. At
            this point any continuous path connecting 0 with $z_l$
            is in $\G.$
    \end{proof}
Set
    \begin{equation}\label{eq:cl}
        c_{mp}:=\inf_{\gamma\in\Gamma}\max_{t\in[0,1]}
        J(\gamma(t)).
    \end{equation}

    \begin{lemma}\label{le:cmp}
        The level $c_{mp}$ is positive.
    \end{lemma}
    \begin{proof}
        Of course it is enough to verify the following geometrical
        mountain pass assumptions: there exist $\d,\rho>0$ such that
            \begin{itemize}
                \item $J(u)\ge \d,\hbox{ for all }u\in \mathcal W_r\hbox{ such that }
                \|u\|=\rho$\\
                \item $J(u)\ge 0,\hbox{ for all }u\in \mathcal W_r\hbox{ such that }
                \|u\|\le\rho.$
            \end{itemize}
By (\ref{phi2}), \ref{en:lupq} of Proposition \ref{pr:lplq} and
since $\W\hookrightarrow L^{p^*}(\RN)$, we have that, if $\|u\|$ is sufficiently small (note that $p<q$)
\begin{align*}
J(u) & \ge c_1 \int_{\L_{\n u}^c} |\n u|^q +  c_2 \int_{\L_{\n u}}
|\n u|^p + (1-\eps) \irn G_2(u) - C_\eps \irn |u|^{p^*}
\\
& \ge c \max \left( \int_{\L_{\n u}^c} |\n u|^q , \int_{\L_{\n u}}
|\n u|^p \right) + c \irn |u|^\a -C_\eps \irn |u|^{p^*}
\\
& \ge c\Big[\|\n u\|_{p,q}^q +\|u\|^\a_{\a} -\|u\|^{p^*}_{p^*}\Big]
\\
& \ge c \Big[\| u\|^{\max\{\a,q\}}  -\|u\|^{p^*} \Big].
\end{align*}
Taking respectively $\|u\|=\rho$ or $\|u\|\le \rho$ with $\rho>0$
sufficiently small we conclude.
    \end{proof}
    We introduce the following auxiliary functional on the space
    $\R\times\mathcal W_r$
        \begin{equation*}
            \widetilde J(\t,u)=\frac{e^{N\t}}{2}\irn \phi(e^{-2\t}|\n
            u|^2)-e^{N\t}\irn G(u).
        \end{equation*}
In analogy with
    $\G$ and $c_{mp},$ we define
        \begin{equation*}
            \widetilde\G=\{\tilde\gamma\in C([0,1],\R\times\mathcal W_r)\mid \tilde\gamma(0)=(0,0),
            \widetilde J(\tilde\gamma(1))< 0\}
        \end{equation*}
    and
        \begin{equation*}
            \widetilde {c}_{mp}:=\inf_{\tilde \g
\in\widetilde\G}\sup_{t\in[0,1]}\widetilde J(\tilde\g(t)).
        \end{equation*}
\begin{proposition}\label{pr:geom}
    The functional $\widetilde J$ verifies the geometrical
    assumptions of the mountain pass theorem, so that $\widetilde
    c_{mp}$ is the mountain pass level. Moreover $c_{mp}=\widetilde
    c_{mp}.$
\end{proposition}
    \begin{proof}
        We estimate the functional $\widetilde J.$ Since $\phi$ is
        increasing in $\R_+,$ by similar computations as those in Lemma \ref{le:cmp}, for small $\|u\|$ we have:
            \begin{align}
                \widetilde J(\t,u)&\ge\frac{e^{N\t}}{2}\irn
                \phi(e^{-2|\t|}|\n u|^2)-e^{N\t}\irn G(u)\\
                &\ge c e^{N\t}
                \Big[e^{-q|\t|}\|\n u\|_{p,q}^q +\|u\|^\a_{\a}
                -\|u\|^{p^*}_{p^*}\Big]\\
                &\ge c e^{N\t}\Big[e^{-q|\t|}\| u\|^{\max\{\a,q\}}  -C\|u\|^{p^*}
                \Big].
            \end{align}
        So we deduce that there exists $\d>0$ such that $\widetilde I(\t,u)$ is
        nonnegative if $\sqrt{\t^2+\|u\|^2}\le\d$, and it is
        positive for $\sqrt{\t^2+\|u\|^2}=\d.$\\
        As in Lemma \ref{le:gamma} we can prove the existence of $(\bar\t,\bar
        u)\in \R\times \mathcal W_r$ for which $\widetilde J(\bar\t,\bar
        u)<0.$

        Finally observe that, since for any $\g\in\G$ we have that $\tilde\g=(0,\g)\in\widetilde
    \G$ and $J\circ \g=\widetilde J\circ\tilde \g,$ certainly $\widetilde c_{mp}\le
    c_{mp}.$ Now, suppose $\tilde \g(\cdot)=(\t(\cdot),\g(\cdot)) \in \widetilde\G.$ Then
    if we set $\eta(t)(\cdot)= \g(t)(e^{-\t(t)}\cdot),$ we have that
    $\eta\in\G$ and $J\circ \eta=\widetilde J\circ\tilde \g.$ So, since $c_{mp}\le\widetilde
    c_{mp},$ we conclude that the two values coincide.
    \end{proof}

    Now we are ready to prove the following fundamental result

    \begin{proposition}\label{pr:psseq}
        There exists a sequence $(u_n)_n$ in $\mathcal W_r$
        satisfying 1, 2 and 3 in Definition \ref{def}, being
        $F=\mathcal Y_r,$ $E=\mathcal W_r$ and $I=J$.
    \end{proposition}

    \begin{proof}
    By standard arguments related with the Ekeland principle, as in \cite{HIT} we can get
    a Palais Smale sequence $(\t_n,u_n)_n$ for
    the functional $\widetilde J$ at the level $\widetilde c_{mp}$
    such that $\t_n\to 0.$\\
    So, from $\widetilde J(\t_n,u_n)\to \widetilde c_{mp},$ $\frac{\partial \widetilde J}{\partial
    u}(\t_n,u_n)\to 0$ in $\mathcal W_r'$ and $\frac{\partial \widetilde J}{\partial
    \t}(\t_n,u_n)\to 0,$ we respectively
    have
        \begin{align}
            &e^{N\t_n}\left(\frac {1}2 \irn \phi(e^{-2\t_n}|\n
            u_n|^2)-\irn G(u_n)\right)\to \widetilde c_{mp}\label{eq:mp}\\
            &e^{N\t_n}\left(\irn\phi'(e^{-2\t_n}|\n u_n|^2)e^{-2\t_n}|\n
            u_n|^2-\irn g(u_n)u_n\right) =o_n(1)\|u_n\|\label{eq:diff}\\
            &e^{N\t_n}\left(\frac N 2 \irn \phi(e^{-2\t_n}|\n
            u_n|^2)\right.\nonumber\\
            &\qquad\qquad\left.-\irn \phi'(e^{-2\t_n}|\n u_n|^2)e^{-2\t_n}|\n
            u_n|^2-N\irn G(u_n)\,dx\right)\to 0\label{eq:poho}
        \end{align}
    Comparing \eqref{eq:mp} with \eqref{eq:poho} we deduce the
    following inequality
        \begin{equation*}
            \frac {e^{N\t_n}}{N} \irn \phi' (e^{-2\t_n}|\n
            u_n|^2)e^{-2\t_n}|\n u_n|^2\to\widetilde c_{mp}
        \end{equation*}
    which is equivalent to
        \begin{equation*}
            \frac 1 N\irn \phi' (|\n
            \tilde u_n|^2)|\n \tilde u_n|^2\to\widetilde c_{mp}
        \end{equation*}
    where $\tilde u_n(\cdot)= u_n(e^{-\t_n}\cdot).$
    By convexity, we know that $0\le \frac 1 2 \phi(t^2)\le\phi'(t^2)t^2,$
    so from the previous inequality we deduce that $\big(\irn \phi(|\n
    \tilde u_n|^2)\big)_n$ is bounded. Assumption \eqref{phi2} and property
    \ref{en:lupq} in Proposition \ref{pr:lplq} imply that
    $(\tilde u_n)_n$ is bounded in $\mathcal Y_r.$ Finally observe that from
    \eqref{eq:mp} we have that $J(\tilde u_n)\to\widetilde c_{mp}$
    and since $\frac{\partial \widetilde J}{\partial
    u}(\t_n,u_n)\to 0,$ we have
        \begin{equation}\label{eq:diff0}
            e^{N\t_n}\Big(\n\cdot\phi'(|\n \tilde u_n|^2)e^{-\t_n}\n\tilde
            u_n +g(\tilde u_n)\Big)\to 0\quad\hbox{in } \mathcal
            W_r'.
        \end{equation}
    Taking into account that $\t_n\to 0,$ from \eqref{eq:diff0} we
    deduce that $J'(\tilde u_n)\to 0$ in $\mathcal W_r'.$ Then $(\tilde u_n)_n$ satisfies 1, 2, and 3 of Definition
    \ref{def}.
    \end{proof}
    We conclude with the proof of our main Theorem

    \begin{proof}[Proof of Theorem \ref{main}]
        Let $(u_n)_n$ be a sequence as in Proposition \ref{pr:psseq}.
        By Theorem \ref{le:ps}, we can extract a subsequence,
        relabeled $(u_n)_n,$ strongly convergent to some
        $u_0\in\mathcal W_r.$ Finally, it is enough to observe
        that, by Lemma \ref{le:cmp} and Proposition \ref{pr:geom},
        $u_0\neq 0.$ Moreover $u_0\ge 0$ by definition of $g_1.$
    \end{proof}

\end{document}